\date{}

\date{}

\title{Factor of iid's through stochastic domination}
\author{\'Ad\'am  Tim\'ar }

\documentclass[11pt,naturalnames]{article}
\renewcommand\footnotemark{}
\usepackage[ruled,vlined]{algorithm2e}
\usepackage{amsmath}
\usepackage{amsthm}
\usepackage{amsfonts}
\usepackage{graphicx}
\newif\ifhyper\IfFileExists{hyperref.sty}{\hypertrue}{\hyperfalse}
\ifhyper\usepackage{hyperref}\fi
\usepackage{enumitem}
\usepackage{subfig}
\usepackage{caption}
\usepackage{keyval}
\usepackage[margin=1in]{geometry}
\usepackage{setspace}
\usepackage[displaymath, mathlines,pagewise]{lineno}
\theoremstyle{definition}
\newtheorem{example}{Example}[section]
\newtheorem{theorem}{Theorem}
\newtheorem{corollary}[theorem]{Corollary}
\newtheorem{lemma}[theorem]{Lemma}

\newtheorem{remark}[theorem]{Remark}

\newtheorem{definition}{Definition}

\def \proof {{ \medbreak \noindent {\bf Proof.} }}
\def\proofof#1{{ \medbreak \noindent {\bf Proof of #1.} }}
\def\proofcont#1{{ \medbreak \noindent {\bf Proof of #1, continued.} }}
\def\supp{{\rm supp}}
\def\max{{\rm max}}
\def\min{{\rm min}}
\def\dist{{\rm dist}}
\def\Aut{{\rm Aut}}
\def\id{{\rm id}}
\def\Stab{{\rm Stab}}

\begin{document}
\maketitle
\let\thefootnote\relax\footnotetext{\footnotesize{Partially supported by Icelandic Research Fund grant number 239736-051 and the ERC Consolidator Grant 772466 ``NOISE''.}}

\bigskip

\def\eref#1{(\ref{#1})}
\newcommand{\Prob} {{\bf P}}
\newcommand{\calC}{\mathcal{C}}
\newcommand{\calP}{\mathcal{P}}
\newcommand{\calQ}{\mathcal{Q}}
\newcommand{\calO}{\mathcal{O}}
\newcommand{\calF}{\mathcal{F}}
\newcommand{\Z}{\mathbb{Z}}
\newcommand{\N}{\mathbb{N}}
\newcommand{\HH}{\mathbb{H}}
\newcommand{\Rr}{\mathbb{R}^3}
\newcommand{\h}{\mathcal{H}}
\def\diam{\mathrm{diam}}
\def\length{\mathrm{length}}
\def\ev#1{\mathcal{#1}}
\def\Isom{{\rm Isom}}
\def\Re{{\rm Re}}
\def \eps {\epsilon}
\def \P {{\Bbb P}}
\def \E {{\Bbb E}}
\def \proof {{ \medbreak \noindent {\bf Proof.} }}
\def\proofof#1{{ \medbreak \noindent {\bf Proof of #1.} }}
\def\proofcont#1{{ \medbreak \noindent {\bf Proof of #1, continued.} }}
\def\supp{{\rm supp}}
\def\max{{\rm max}}
\def\min{{\rm min}}
\def\dist{{\rm dist}}
\def\Aut{{\rm Aut}}
\def\id{{\rm id}}
\def\Stab{{\rm Stab}}
\def\T{{\cal T}}
\def\B{{\cal B}}

\newcommand{\lra}{\leftrightarrow}
\newcommand{\xlra}{\xleftrightarrow}
\newcommand{\xnlra}{\xnleftrightarrow}
\newcommand{\pc}{{p_c}}
\newcommand{\pt}{{p_T}}
\newcommand{\ptk}{{\hat{p}_T}}
\newcommand{\pl}{{\tilde{p}_c}}
\newcommand{\pe}{{\hat{p}_c}}
\newcommand{\pr}{\mathrm{\mathbb{P}}}
\newcommand{\pp}{\mu}
\newcommand{\ex}{\mathrm{\mathbb{E}}}
\newcommand{\ee}{\mathrm{\overline{\mathbb{E}}}}

\newcommand{\om}{{\omega}}
\newcommand{\ebd}{\partial_E}
\newcommand{\ivbd}{\partial_V^\mathrm{in}}
\newcommand{\ovbd}{\partial_V^\mathrm{out}}
\newcommand{\q}{q}
\newcommand{\RR}{\mathcal{R}}

\newcommand{\CC}{\Pi}
\newcommand{\BB}{\Pi}

\newcommand{\A}{\mathcal{A}}
\newcommand{\cc}{\mathbf{c}}
\newcommand{\pa}{{PA}}
\newcommand{\degi}{\deg^{in}}
\newcommand{\dego}{\deg^{out}}
\def\Pn{{\bf P}_n}

\newcommand{\R}{\mathbb R}
\newcommand{\F}{F}
\newcommand{\FF}{\mathfrak{F}}
\newcommand{\Can}{\rm Can}
\newcommand{\Vol}{\mathrm Vol}

\def\UST{{\rm UST}}
\def\Gstar{{{\cal G}_{*}}}
\def\Gstarstar{{{\cal G}_{**}}}
\def\Gstarplus{{{\cal G}_{*}^\frown}}
\def\Rel{{\cal R}}
\def\Comp{{\rm Comp}}
\def\calG{{\cal G}}
\def\calM{{\cal M}}
\def\calN{{\cal N}}
\def\omps{{\omega_\delta^\eps}}
\def\FKI{{\phi^f_{G,p,0}}}
\def\FKIn{{\phi_{G_n,p,0}}}
\def\FKIf{{\phi^f_{G,p,0}}}
\def\FKIw{{\phi^w_{G,p,0}}}
\def\loop{{\Prob_{G,x}}}
\def\loopn{{\Prob_{G_n,x}}}

\begin{abstract}
We develop a method to prove that certain percolation processes on amenable random rooted graphs are factors of iid, given that the process is a monotone limit of random finite subgraphs that satisfy a certain independent stochastic domination property. 
Among the consequences is the previously open claim that the Uniform Spanning Forest is a factor of iid for recurrent graphs, and that it arises as a finitary factor.
\end{abstract}

\bigskip

\section{Introduction}

%Consider some possibly infinite graph $G=(V(G),E(G))$, and let $\omega$ be some random subgraph or random labelling of the vertices or edges. A question of general interest is weather $\omega$ can be attained as a {\it factor of iid}, that is, if one can find a function from iid uniform $[0,1]$ labellings of $V(G)$ to another labelling so that the function is defined locally, such that the latter is distributed as $\omega$. A necessary condition to be a factor of iid is that $\omega$ has a distribution invariant under the automorphisms of $G$. 
%One example to think of is the FK model, which include, with suitably chosen parameters, Bernoulli percolation, a version of the Ising model, or (for limiting parameter values) the Free Uniform Spanning Forest.
%More generally, one may take $G$ to be a unimodular random graph, where the role of transitivity is taken by a certain invariance under rerooting, and invariant random processes are replaced by processes that are jointly unimodular with $G$.
For a given graph $G$, consider an iid labelling of its vertices by uniform $[0,1]$ random variables. 
A factor of iid process on $G$ is a new labelling of the graph where the label of each vertex can be encoded from the iid labels in its neighborhood, and the same coding rule is used for every vertex. We will give a formal definition later.
For example, the Ising model on $\Z^d$ with all-plus boundary conditions is always a factor of iid \cite{A}, but with wired boundary conditions it is a factor of iid only if there is a unique Gibbs measure.
See the Introduction of \cite{ARS} for a list of factor of iid models and references. 
%A recent striking result of Nam, Sly and Zhang shows that for ... the Ising model is not a factor of iid \cite{NSZ}.
%Factor of iid constructions are coming from ergodic theory, where a factor construction (constructing one process as a local function of the other) is a fundamental building block for a hierarchy (and equivalence) between processes. In this regards, the iid uniform $[0,1]$ labelling is a process that has ``as much independence as possible'', has the largest entropy in the setting of an amenable $G$ (where entropy is defined), 
Among the benefits of factor of iid constructions are that they go through when taking local limits of graphs, hence they are well suited for sofic approximation. From the algorithmic point of view, factor of iid constructions can be viewed as simple models of computing with parallel processors, a natural distributed algorithm. As we will see, a hierarchical categorization of random processes has been introduced, based on whether they are factors of iid with further restrictions (the coding using a fixed bounded neighborhood, using a random but bounded neighborhood...).
%they preserve indistinguishability of clusters (sajat cikkem).

The Wired Uniform Spanning Forest (WUSF) of a transient graph can be attained as a factor of iid: it follows from \cite{BLPS} that the cycle-popping procedure, which is also underlying ``Wilson's algorithm rooted at infinity'', defines a factor map from an iid labelling to the WUSF. The question of whether the uniform spanning tree on an infinite recurrent $G$ is a factor of iid remained open, apart from the case of Cayley graphs \cite{LT}; see Question 6.1 in \cite{ARS}. Other examples where the study of USF on recurrent graphs have shown to be more difficult than on transient ones include the question of their one-endedness, where this was the last unresolved case one until recently \cite{EH}.

Note that on a recurrent, or more generally, an invariantly amenable graph, the free and the wired Uniform Spanning Forest are the same. To indicate this equality, but keep in mind that we are dealing with infinite graphs, we will denote it by USF for such graphs, even though the USF always has a unique component on recurrent graphs. 

\begin{theorem}\label{recurrentUSF}
Let $G$ be a unimodular random graph that is almost surely recurrent. Then the USF of $G$ is a factor of iid.
\end{theorem}

The theorem is a part of Corollary \ref{usf_amen}. The full proof is given in Section \ref{sec.usf}.

We mention that the question whether the Free Uniform Spanning Forest is a factor of iid in full generality is open. It is of special importance, because of its connections to the fixed-price conjecture (as explained e.g. in \cite{HP}).

A special class of fiid processes is formed by those where the factor labelling of a vertex can be determined without any error from a large enough neighborhood of that vertex, where the size of the necessary neighborhood possibly depends on the iid labels in it, similarly to how stopping times are defined. Such fiid processes are called finitary factor of iid or ffiid. A further restriction is when the iid labels are not Lebesgue [0,1], but rather we have a constant number of random bits on every vertex. A finitary factor from such an iid labelling is called a finite-valued finitary factor of iid.

\begin{theorem}\label{fvffiid}
(1) The USF on an invariantly amenable unimodular random graph is a finite-valued finitary fiid. \\
(2) The Ising model on $\Z^d$ when there is a unique Gibbs measure is a finite-valued finitary fiid.
\end{theorem}
Part (1) is new, while part (2) has been essentially known, as explained to us by Yinon Spinka. The existence of a finitary fiid construction in part (2) essentially follows from \cite{BS}. Meyerovitch and Spinka \cite{MSp} proved that a finite-entropy, countable-valued finitary factor of iid process on an amenable graph can always be constructed as a finite valued factor of iid, which completes (2). Our original manuscript contained a direct proof of the finite-valued ffiid's in Theorem \ref{fvffiid}, but after learning about \cite{MSp}, we removed that argument and only keep the part for (1) that ensures a finitary construction. See also Remark \ref{history_finitary}.

%hile part (2) was shown for supercitical temperature by van den Berg and Steif \cite{BS}, and remained open for critical temperature (see Remark 1.4 in \cite{Sp2}).
%We believe that this theorem is true for any invariantly amenable URG $G$ instead of $\Z^d$, but we did not pursue this direction. What is missing for such an argument is a {\it finite-valued} finitary fiid version of Lemma \ref{Zd}. 

Section \ref{definitions} contains the standard definitions that we will rely on.
In Section \ref{sec.usf} we present our definition of a ``compatible monotone limit'', show that certain models satisfy it, and prove how it implies the existence of fiid codings.

\section{Definitions}\label{definitions}

Denote the set of locally finite graphs with a distinguished root vertex and up to rooted isomorphisms by ${\cal G}_{*}$. Let ${\cal G}_{**}$ be the set of locally finite graphs with a distinguished ordered pair of vertices, up to isomorphisms preserving the ordered pair. 
A {\it unimodular random graph (URG)} is a probability measure $\mu$ on ${\cal G}_{*}$ satisfying
\begin{equation}
\int \sum_{x\in V(G)} f(G,o,x) d\mu ((G,o))=\int \sum_{x\in V(G)} f(G,x,o) d\mu ((G,o))
\label{eq:MTP}
\end{equation}
for every Borel $f:{\cal G}_{**}\to [0,\infty]$. This latter equation is usually referred to as the Mass Transport Principle. See \cite{AL} or \cite{Cu} for more on URG's.
By a slight abuse of terminology we will often refer to the random graph $(G,o)$ as a URG, and also, in general we will sometimes refer to probability measures on graphs as random graphs.
Also, we may hide the root in our notation and call $G$ a URG.

One can decorate the vertices and edges of a URG with some marking (elements of some metric space $X$, often finite), and then it is possible to represent a percolation (i.e., random subgraph), coloring, collection of subgraphs etc. of a URG by a marked version of the URG. We say that a percolation or other decoration of the URG $G$ is {\it invariant}, if this marked graph also satisfies \eqref{eq:MTP} in the extended space of marked graphs. When the URG is supported on some fixed quasi-transitive unimodular graph $G$ (such as a Cayley graph) with a random marking, then this definition of invariance is equivalent to the usual invariance under the automorphisms of $G$.
 
Let $G=(V,E)$ be a unimodular quasi-transitive graph, and denote by $2^G$ the set of its subgraphs. 
We will consider {\it random subgraphs} that are either edge sets or vertex sets of $G$, and think about it as a random object and as a probability measure as well. 
If $(X_v)_{v\in V}$ is a collection of iid Lebesgue[0,1] random labels
on the vertices and $\mu$ is a random subgraph of $G$, say that $\mu$ is a {\it factor of iid (fiid)}, if there is a {\it coding} map (also called an {\it fiid rule}) $\phi$ from
$[0,1]^V$ to $2^G$ that is measurable and equivariant with the automorphisms of $G$, and such that the push-forward of the iid Lebesgue labels by $\phi$ is $\mu$. A more hands-on way to put this definition is that for any $\eps>0$ there is an $R$ such that by looking at the $X_v$ in the $R$-neighborhood of $o$ we can tell whether $o$ and the edges incident to it should be in the random subgraph given by $\mu$, up to an error of probability less than $\eps$, and in telling so we only need information in this $R$-ball up to rooted isomorphisms. This latter definition can be applied to any unimodular random graph without the assumption of quasi-transitivity. If one can almost surely stop at an $R$ such that the status of $o$ and its edges can be determined without any error, and that only information in the $R$-ball is used to determine that we stop, 
then we call $\mu$ a {\it finitary factor of iid}. Finally, if instead of Lebesgue[0,1] we label every vertex by $k$ iid Bernoulli($1/2$) random variables (``random bits'') for some uniform finite $k$, then the existence of a finitary fiid coding means that $\mu$ is a {\it finite valued finitary fiid (fv fiid)}.

A unimodular random graph is {\it invariantly amenable} (also called hyperfinite) if there exists 
an increasing sequence of induced subgraphs $\Gamma_n$ of $G$ ($n=1,2\ldots$) whose union is $G$, the collection $(G,\Gamma_1,\Gamma_2,\ldots)$ is jointly unimodular, and such that all components of every $\Gamma_n$ are finite. We call such a sequence a {\it hyperfinite exhaustion}. See \cite{AL} for proofs of why this is equivalent to alternative common definitions. As proved in Lemma 2.3 of \cite{T}, if $G$ is invariantly amenable, then one is also able to construct the above $\Gamma_n$ as a factor of iid. The proof also implies the following.

\begin{lemma}\label{hyperfinite_fiid}
If $G$ is an invariantly amenable unimodular random graph, then there exists a sequence of induced subgraphs $\Gamma_n$ such that every component of $\Gamma_n$ is finite, $\Gamma_n\subset\Gamma_{n+1}$, $G=\cup\Gamma_n$, and $\Gamma_n$ is a finitary factor of iid for every $n$.
\end{lemma}

If $H$ is a subgraph of $K$, and $\mu$ is a probability measure on $2^K$, then $\mu|_H$ is the probability measure corresponding to the random graph given by $\mu$ restricted to $H$. If $H_1,\ldots, H_k$ are pairwise vertex-disjoint graphs and $\mu_i$ is the distribution of a random subgraph of $H_i$, then denote by $\otimes\mu_{i}$ the union of the random graphs sampled independently by the $\mu_i$ over $i\in\{1,\ldots,k\}$.

%\begin{definition}\label{fiid}{\bf (Factor of iid)} EZ SZO SZERINT A FA-FAKTOROS CIKKBOL VAN
%Let $X$ be some separable metric space.
%A {\it factor map} is a Borel measurable function $f : {\cal G}_* \to X$. Let G be a random graph, let  $\lambda : V (G)\to [0, 1]$ be iid Lebesgue[0,1] {\it labels} on its vertices, and let $G(\lambda)$ be the random labeled graph given by the labels $\lambda$. The collection of random variables $\bigl(f((G(\lambda),v))\bigr)_{v\in V(G)}$ is called a {\it factor of iid (fiid)} process if $f$ is a factor map.
%\end{definition}

%Ha lesz hiperbolikus grafokrol fejezet, azt felvezethetne valami hasonlo: free Loop O(1) model with no external field, invar nonamen - erre nem ismert, h fiid-e hacsak nem teljesul, hogy az FK-Ising p=2x/(1+x) barmely csucsan veges sok geod kor megy at

From now on $G=(G,o)$ will always denote an invariantly amenable unimodular random graph (URG). For convenience, we will assume that the iid labels are coming from a suitable product set 
$\Xi=\Xi_0\times\Xi_1\times\ldots$ with some product probability measure on it. %We will always reserve the variables from $\Xi_0$ for symmetry breaking, meaning that for the fiid constructions we will start 

\section{Monotone weak limits as factor of iid}\label{sec.usf}

%We can take $\Xi_i=[0,1]$ with the uniform measure, which is itself isomorphic to $\Xi$ as a measure space. 

\begin{definition}\label{monotone}
Let $\mu$ be a random subgraph of $G$. 
%For simplicity, we will speak about bond percolation, but everything would work for site percolation as well. 
Say that $\mu$ is a {\it compatible monotone decreasing limit} if for every finite connected graph $H$ one can define a random subgraph $\mu_H$ of $H$ in such a way that 
\begin{enumerate}
\item whenever $G_n$ is a finite exhaustion of $G$, $\mu_{G_n}$ weakly converges to $\mu_G$;
\item if $H_1,\ldots, H_k$ are pairwise vertex-disjoint connected subgraphs of the finite connected $H$, then $\otimes\mu_{H_i}$ stochastically dominates $\mu_{H}|_{\cup H_i}$;
\item $\mu_H$ depends only on the isomorphism class of $H$ in the sense that if $H$ and $H'$ are subgraphs and $\iota:H\to H'$ is an isomorphism, then the pushforward of $\mu_H$ by $\iota$ is $\mu_{H'}$.
\end{enumerate}
We define a {\it compatible monotone increasing limit} similarly, with the direction of stochastic domination changed in (2). A {\it compatible monotone limit} is a compatible monotone decreasing or monotone increasing limit.
\end{definition}

\begin{example}
The Free Uniform Spanning Forest (FUSF) of a URG is an example of a compatible monotone decreasing limit. This is clear when $k=1$ in the above definition, from the standard proof of the existence of these forests, using monotone couplings based on Rayleigh's monotonicity law. For $k>1$, one can add paths between the $H_i$ to connect them in a tree-like way, and apply the previous argument to this connected subgraph of $H$.
\end{example}

\begin{example}
The WUSF is a compatible monotone increasing limit on any graph. For this, we will use a technical modification of Definition \ref{monotone}, namely, we require 2. to hold only when the minimal distance between the $H_i$ is at least 2 and $H\setminus\cup H_i$ is connected. (This more general definition could be used in each of our applications without any complication; we prefered to stick to the more natural one in Definition \ref{monotone}, because in most cases it is sufficient.) Denote by $(\cup H_i)^w$ the graph where we replace every edge of $G$ having a single endpoint $x$ in some $H_i$ by an edge $\{x,w\}$, where $w$ is an extra vertex added to $\cup H_i$, and define $H_i^w$ to be $H_i$ with a wired boundary. Then, with notation as in Definition \ref{monotone}, 
$$\UST (H^w)|_{\cup H_i}\succeq \UST ((\cup H_i)^w)|_{\cup H_i}=\otimes \UST (H_i^w),
$$
where the stochastic domination follows because $H^w$ can be obtained from $(\cup H_i)^w$ by contracting edges outside of $\cup H_i$.
\end{example}

\begin{example}
Consider the Ising model with the all-plus boundary condition, and the random subgraph given by vertices with plus spin. This is a compatible monotone decreasing limit, as follows from the fact that on a fixed finite graph the Ising model with any boundary condition is stochastically dominated by that with the plus boundary condition. Likewise, the Ising model with minus boundary conditions is a compatible monotone increasing limit. 
\end{example}

\begin{example}
Similarly to the previous example, Fortuin-Kasteleyn random cluster models of parameter $q\geq 1$ with free (respectively: wired) boundary condition are compatible monotone decreasing (respectively: increasing) limits. 
\end{example}

\begin{theorem}\label{main}
Let $G$ be an invariantly amenable URG. Suppose that the random subgraph $\mu_G$ is a monotone weak limit. Then $\mu_G$ is a factor of iid.
\end{theorem}

\begin{proof}
We prove the claim for a compatible monotone decreasing limit. The monotone increasing case would work the same way.

Consider a factor of iid sequence of subgraphs $\Gamma_n$ of $G$ ($n=1,2\ldots$) such that $\Gamma_n\subset\Gamma_{n+1}$, their union is $G$, and every component of any $\Gamma_n$ is finite. Such a sequence exists by Lemma 2.3 of \cite{T} (Lemma \ref{hyperfinite_fiid}). Use the labels from $\Xi_0$ to define this.

For every finite $H$ fix some measurable map $\phi_1$ from $\Xi_1^H$ to the subgraph of $H$, in such a way that the iid measure on the labels from $\Xi_1$ is mapped into the measure $\mu_H$ as in Definition \ref{monotone}, and such that $\phi_1$ commutes with the automorphisms of $H$. 
For each component $H$ of $\Gamma_1$ assign the random $\omega_H$ using $\phi_1$ on $H$, and let $\omega_{\Gamma_1}$ be the union of these disjoint subgraphs $\omega_H$ over the $H$. Suppose that $\omega_{\Gamma_{n-1}}$ has been defined, it is a random subgraph of $\Gamma_{n-1}$. For every finite graph $H$ and $H_1,\ldots, H_k$ pairwise disjoint subsets of $H$, define a monotone coupling between $\omega_H$ and $\cup\omega_{H_i}$. Such a coupling exists by (2) and Strassen's theorem. So we can randomly remove some of the edges of $\cup\omega_{H_i}$, and add some edges of $H\setminus \cup E({H_i})$, and get a subgraph of distribution $\omega_H$ in $H$. Now for every finite $H$, conditional on $\cup\omega_{H_i}$, fix some rule of removing and adding the edges as some function from $\Xi_n^H$ (we can call $\phi_n$ the function describing this rule). This rule should be the same for isomorphic collections $(H,H_1,\dots,H_k)$.
Apply this rule to each of the components of $\Gamma_n$, to define $\omega_{\Gamma_n}$ as a factor of iid.

Every edge can change status at most twice during this procedure: the first change may happen if it becomes part of $\omega_{\Gamma_n}$ for some $n$, and the second one may happen if it gets removed in some later stage (and then it stays like that). Hence the almost sure limit $\omega_G$ of $\omega_{\Gamma_n}$ is a factor of iid.
\qed
\end{proof}

\medskip

\begin{corollary}\label{usf_amen}
The USF of an invariantly amenable URG is a factor of iid. In particular, this holds when the URG is almost surely recurrent.
\end{corollary}
The last claim was not known, see Question 6.1 by Angel, Ray and Spinka in \cite{ARS}. 
%Besides the common theme of factor of iid representations, the reason it is connected to the rest of our paper and theirs is the following. Generating the Loop $O(1)$ model as factor of iid is possible through a factor of iid locally finite generating set of the cycle space, see .... One way to obtain it, when the underlying graph is amenable, is via choosing a factor of iid one-ended spanning tree. (Here we need to assume that $G$ itself is not two-ended.) If the USF was 

%\begin{remark}
%One possible criticism to our proof of Theorem \ref{recurrentUSF} could be that the resulting factor of iid mapping is somewhat implicit, unlike the construction for the transient case where the cycle popping procedure through Wilson's algorithm makes the mapping more transparent. While we agree that using Wilson's algorithm makes that construction more canonical in a sense, we also remark that for that argument to work, a ``legal ordering'' of all the vertices was needed (Theorem 4.1 in \cite{ARS}), and the use of transfinite induction seems to be necessary for the rigorous proof.
%\end{remark}

\begin{remark}\label{remark7}
The construction in the proof is {\it $\Gamma_n$-local}, by which we mean that conditioned on $\Gamma_n$, for every component $K$ of $\Gamma_n$ we defined $\omega_{\Gamma_n}|_K$ independently from the others, and using only the iid labels in $K$.
%The construction in the proof is {\it $\Gamma_n$-local}, by which we mean the following. Using the first labels from $\Xi_0$ to get a labelled version of $G$ that has almost surely no symmetries. Then use labels from $\Xi_2\times\Xi_4\times\ldots$ to construct the sequence $\Gamma_n$ as in Lemma \ref{hyperfinite_fiid}, and use $\Xi_1\times \Xi_3\times\ldots$ to construct $\omega_{\Gamma_n}$. Then, conditioned on $\Gamma_n$, for every component $K$ of $\Gamma_n$ we defined $\omega_{\Gamma_n}|_K$ independently from the others, and using only the $\Xi_1\times \Xi_3\times\ldots$-labels in $K$.
\end{remark}

\section{Finitary fiid's for amenable graphs}\label{section_finitary}

\begin{theorem}\label{finitary_Ising}
Let $G$ be an invariantly amenable URG, and suppose that there is a unique Gibbs measure for the Ising model. Then it is a finite valued finitary factor of iid.
\end{theorem}

\begin{remark}\label{history_finitary}
The previous theorem has been known, we only present an alternative proof. %See the explanation after Theorem \ref{fvffiid}. 
The subcritical case has been settled by van den Berg and Steif in \cite{BS} (and strengthened in \cite{Sp2} where finite expected coding volume is proved), even with a choice of a finite-valued ffiid. They also showed, based on the method of \cite{MS}, that in the phase of multiple Gibbs measures there is no finitary coding. From \cite{Y}, \cite{ADS} and \cite{AF} we know that the Ising model on $\Z^d$ has a unique Gibbs measure at critical temperature, and the method in \cite{BS} applies to that case to get a finitary fiid, which then can be chosen to be finite valued by \cite{MSp}. %Compare it to the result in \cite{BS} that any factor of iid coding of the critical Ising model necessarily has infinite expected coding volume (meaning that if $R\in [0,\infty]$ is the smallest radius needed for coding, then $R^d$ has infinite expectation).
\end{remark}

\begin{proof}
Consider the finitary fiid hyperfinite exhaustion $\Gamma_n$ as in Lemma \ref{hyperfinite_fiid}. Let $\calP_n$ be the set of connected components of $\Gamma_n$.

For a given finite subgraph $H\subset G$ consider the Ising model on $H$ with all-plus boundary conditions (respectively: all-minus boundary conditions), and define the vertices that got a positive spin to be $\omega_H^+$ (respectively: $\omega_H^-$). Then, for $H\subset K$ finite subgraphs, we have 
$$\omega_H^+\succeq \omega_K^+|_H \succeq \omega_K^-|_H \succeq \omega_H^-,$$
so similarly to the proof of Theorem \ref{main}, if a monotonous coupling of
$\cup_{H\in\calP_n}\omega_H^+$ and $\cup_{H\in\calP_n}\omega_H^-$ have been constructed (so that $\cup_{H\in\calP_n}\omega_H^+\supset\cup_{H\in\calP_n}\omega_H^-$), then one can remove vertices from $\cup_{H\in\calP_n}\omega_H^+$ and also add vertices from $\Gamma_{n+1}\setminus\Gamma_n$
to get $\cup_{K\in\calP_{n+1}}\omega_K^+$
and add vertices to $\cup_{H\in\calP_n}\omega_H^-$ to get $\cup_{K\in\calP_{n+1}}\omega_K^-$ in a way that $\cup_{K\in\calP_{n+1}}\omega_K^+\subset\cup_{K\in\calP_{n+1}}\omega_K^-$, 
and do this for each $K$ following some fixed fiid rule being the same for isomorphic tupples $(K,\{H\subset K, H\in\calP_n\})$. By step $n$ the construction is finitary, because $\calP_n$ is finitary and we only have to use information from within the element $H$ of $\calP_n$ that contains $o$ to be able to tell the status of $o$ in $\omega_H^+$ and in $\omega_H^-$. At some finite $n$, $\omega_H^+(o)=\omega_H^-(o)$, and from this point on it will remain unchanged. 

The assertion that there also exists a finite valued coding follows from finitariness by \cite{MSp}.
\qed
\end{proof}

\begin{theorem}\label{finitary_USF}
Let $G$ be an invariantly amenable URG. Then the USF
is a finite valued finitary factor of iid. 
\end{theorem}

It is enough to prove the existence of a finitary factor of iid representation, which can then be made finite valued by \cite{MSp}, using amenability.

\begin{proof}[1st proof]
This proof only works for the case when $G$ is not two-ended almost surely.
It is known that then all the components of the WUSF are one-ended. Namely, it was shown in \cite{AL} that every tree of the WUSF of a transient URG is almost surely one-ended when the URG has bounded degrees, then in \cite{H} the bounded degree assumption got removed, and finally the claim was shown in \cite{EH} for all recurrent URG's that are not two-ended (that is, one-ended). See this latter for a detailed history and further references.

In the proof of Theorem \ref{main}, suppose that $e$ is an edge and $M$ is the first time when $e\in\Gamma_M$. If $e$ is not in the FUSF$=:\omega_G$ then $\omega_{\Gamma_n}$ for all $n\geq N$, where $N$ is either equal to $M$ and $e\not\in\omega_{\Gamma_N}$, or $N>M$ is the time when $e$ is removed from $\omega_{\Gamma_N}$, i.e., $e\in \omega_{\Gamma_{N-1}}\setminus \omega_{\Gamma_N}$. So one can tell if $e$ is not in $\omega_G$ from looking at the neighborhood that contains its component in $\Gamma_N$.

Now suppose that $e\in\omega_G$. By assumption, the component $C$ of $e$ has one end. Consider the finite component $C_0$ of $C\setminus \{e\}$, and observe that none of the finitely many edges of $\partial C_0 \setminus \{e\}$ is in
$\omega_{\Gamma_n}$ if $n$ is large enough. But once we know the status of these edges, we also know that $e$ has to be in $\omega_G$, because otherwise one of its endpoints would be in a finite component, which is not possible. So from a large enough neighborhood we can tell with certainty that $e\in\omega_G$, by the previous paragraph. 

\medskip

\noindent [2nd proof]
For invariantly amenable URG's the free and the wired uniform spanning forests are the same \cite{BLPS}.
Using that the WUSF is a compatible monotone increasing limit, the FUSF is a compatible monotone decreasing limit, and using that the wired is dominated by the free, we can apply a similar sandwiching argument as in the proof of Theorem \ref{finitary_Ising}, to get that the factor is finitary. Unlike the previous proof, this one does not use any prior knowledge about the number of ends.
\qed
\end{proof}

\begin{remark}\label{remark8}
Similarly to Remark \ref{remark7}, the constructions in Theorem \ref{finitary_Ising} and the second proof of Theorem \ref{finitary_USF} are $\Gamma_n$-local.
\end{remark}

\noindent
{\bf Acknowledgments:} 
I thank Russ Lyons and Yinon Spinka for valuable comments on the manuscript and G\'abor Pete for discussions.

\ \\

{\bf \'Ad\'am Tim\'ar}\\
Division of Mathematics, The Science Institute, University of Iceland\\
Dunhaga 3 IS-107 Reykjavik, Iceland\\
and\\
HUN-REN Alfr\'ed R\'enyi Institute of Mathematics\\
Re\'altanoda u. 13-15, Budapest 1053 Hungary\\
\texttt{madaramit[at]gmail.com}\\

\end{document}